 \documentclass[10pt,a4paper]{amsart}
\usepackage{amscd,amssymb,amsopn,amsmath,amsthm,graphics,amsfonts,enumerate,verbatim,calc}
\usepackage[dvips]{graphicx}

\usepackage{mathpazo}
\usepackage{color}
\usepackage{latexsym}
\usepackage{amsthm,amsfonts,amssymb,mathrsfs}
\usepackage{rotating}
\usepackage[leqno]{amsmath}
\usepackage{xspace}
\usepackage[all]{xy}
\usepackage{longtable}
\headsep=5 mm \numberwithin{equation}{section}
\newtheorem{theorem}{Theorem}[section]
\newtheorem{lemma}[theorem]{Lemma}

\newtheorem{proposition}[theorem]{Proposition}
\newtheorem{corollary}[theorem]{Corollary}

\theoremstyle{definition}
\newtheorem{definition}[theorem]{Definition}
\theoremstyle{remark}
\newtheorem{remark}[theorem]{Remark}
\newtheorem{fact}[theorem]{Fact}
\newtheorem{example}[theorem]{Example}
\newtheorem{observation}[theorem]{Observation}

\newtheorem{question}[theorem]{Question}

\newtheorem{acknowledgement}{Acknowledgement}

\newcommand{\Ass}{\operatorname{Ass}}

\newcommand{\grade}{\operatorname{grade}}
\newcommand{\Soc}{\operatorname{Soc}}

\newcommand{\Assh}{\operatorname{Assh}}
\newcommand{\Spec}{\operatorname{Spec}}

\newcommand{\rad}{\operatorname{rad}}

\newcommand{\Ht}{\operatorname{ht}}

\newcommand{\pd}{\operatorname{p.dim}}

\newcommand{\UFD}{\operatorname{UFD}}

\newcommand{\HH}{\operatorname{H}}
\newcommand{\nil}{\operatorname{nil}}

\newcommand{\e}{\operatorname{e}}

\newcommand{\V}{\operatorname{V}}

\newcommand{\Ext}{\operatorname{Ext}}

\newcommand{\Hom}{\operatorname{Hom}}

\newcommand{\Ann}{\operatorname{Ann}}

\newcommand{\zd}{\operatorname{zd}}

\newcommand{\depth}{\operatorname{depth}}

\newcommand{\U}{\operatorname{U}}

\newcommand{\lo}{\longrightarrow}
\newcommand{\fm}{\mathfrak{m}}
\newcommand{\fp}{\mathfrak{p}}
\newcommand{\fq}{\mathfrak{q}}
\newcommand{\fa}{\mathfrak{a}}

\newcommand{\fn}{\mathfrak{n}}

\begin{document}
	
	\author[]{Mohsen Asgharzadeh}
	
	\address{}
	\email{mohsenasgharzadeh@gmail.com}

	\title[ ]
	{a note on a system of parameters}

	
	\begin{abstract}
	Let  $u$ be in $\fp\in\Assh(R)$. We present several  situations
for which $(0 : u)$ is (not) in an ideal generated by a system of parameters.
An application is given.
\end{abstract}

\maketitle

\section{Introduction}

Let  $(R,\mathfrak m,k)$ be a noetherian   local ring of dimension $d$. We say a sequence
$\underline{x}:=x_1,\ldots,x_d$ of elements of $\fm$ is a system of parameters
if   $\ell( R/\underline{x}R)<\infty $. By $I$ we mean the ideal  generated by a system of parameters.
By $\Assh(R)$ we mean $\{\fp\in\Ass(R):\dim R/\fp=d\}$. Let   $\fp\in\Assh(R)$ and take $u$ be  in $\fp$.  

\begin{question}(See \cite[Question 6.4]{fh}) Can $(0 : u)$ ever
be in  $I$?
\end{question}

 For the motivation,  see   \cite[Introduction]{fh} by Fouli and Huneke.
Their calculations   strongly
suggest the answer is always `no'. For instance, over  $1$-dimensional rings. Also, over Gorenstein rings the answer is  no, because of validity of the monomial conjecture, see Fact \ref{gor}. 
We extend this by dropping the Cohen-Macaulay assumption: 

\begin{observation}
	Question 1.1  is not true over quasi-Gorenstein rings.
\end{observation}
It is easy to see that  Question 1.1  is not true in each of the following  three situations:	 i)    $\frac{k[[ X,Y,Z]]}{J}$ for some unmixed ideal $J$,
	 ii)   Cohen-Macaulay rings of multiplicity two, and
	 iii)  $\frac{k[[ X_1,\ldots,X_n]]}{(f,g)}$ where $(f,g)$ is unmixed. 
	 Due to Observation 1.2, we pay a special attention  to non quasi-Gorenstein rings with nontrivial zero-divisors. This enable us
to check Question 1.1 in some new cases. Here, is a sample:

\begin{observation}
	Let $P$ and $Q$ be two prime ideals of $S:= k[[X_1,\ldots,X_d]]$ generated by linear forms.  Then Question 1.1    is not true over $R:=\frac{S}{PQ}$.
\end{observation}
For a related result concerning  powers of a prime ideal, see Proposition \ref{power}.
These observations have an application, see e.g. Corollary \ref{ap}.   It may be worth
to note that Eisenbud and Herzog 
predicted that product of ideals of height at least two  in a regular ring  is not Gorenstein.
For an important progress, see  \cite{h}. In \S3  we present a connection 
to this problem, see 
Proposition  \ref{catg2} and its corollary. In fact, we give   situations for which a
product of two ideals is  neither Cohen-Macaulay nor quasi-Gorenstein. For instance, see Corollary \ref{gcm}.

For each $n>0$, set $	R_n:= \frac{k[X, Y, Z, W]_{\fm}}{(XY - ZW, W^n, YW)}$. In \S 4, by mimicking from \cite{st}, we
 show  Question 1.1 is not true over $R_n$ if and only if $n=1$. 
The ring $R_2$ is two-dimensional,
generically Gorenstein, Cohen-Macaulay, almost complete-intersection, of type two and of minimal multiplicity three. Also, in \S 4 we present a  ring of multiplicity two equipped with a prime ideal  $\fp\in\Assh(R)$ and  $u\in\fp$ such that $(0 : u)\subset I$.
In the final section    we talk a little about a  question  by Strooker and St\"{u}ckrad: 
	What is the set-theoretic  union of
	all parameter ideals?


\section{Positive side of Question 1.1}
We start by recalling the following well-known facts:

\begin{fact}\label{well} (See \cite[Thorem 14.1]{Mat}) Let $A$ be   a local ring and let $\underline{a}:=a_1,\ldots,a_t$ be a part of system of parameters. Then $\dim A/\underline{a}A=\dim A-t$.
\end{fact} 

\begin{fact}(Fouli-Huneke)\label{1}
	Let $R$ be a $1$-dimensional local ring and  $u$ be in $\fp\in\Assh(R)$.  Then $(0 : u)$ is not in an ideal generated by a system of parameters.
\end{fact}

\begin{proof}
Suppose by way of contradiction that  there is a	parameter ideal $(x)$ such that $(0:u)\subset(x)$.
 Let $r$ be such that $ru=aux$ for some $a$.
		Thus, $ r-ax \in(0:u)\subset(x)$. From this, $r\in(x)$. So, the map $R/(x)\stackrel{u}\lo R/(ux)$ is injective.
		In view of \cite[Theorem 4.1]{fh} we see that $ux$ is   parameter. Since $ux\in\fp\in\Assh(R)$, we get to a contradiction.
\end{proof}

\begin{fact}(Fouli-Huneke and others)\label{gor}
	Let $R$ be a Gorenstein local ring,  $u$ be in $\fp\in\Assh(R)$.  Then $(0 : u)$ is not in an ideal generated by a system of parameters.
\end{fact}

\begin{proof}
Fouli and Huneke remarked that the desired claim follows from the validity of   monomial conjecture.
 Recently,  Andr\'{e} \cite{an} proved this.
\end{proof}

 By $\mu(-)$ we mean the minimal number of elements that needs to generate $(-)$.
 
\begin{observation} \label{em}
	Let $(S,\fn)$ be a regular local ring and $J\lhd S$  be unmixed.  Adopt one of the  following situations:
 i)    $\dim S<4$, or
 ii)  $\mu(J)<3$. 
Then    Question 1.1  is not true over $R:=S/J$.
\end{observation}

\begin{proof}
i) In the light of Fact \ref{1} we  may  assume that $1<\dim R<4$. 
If $\dim R=3$, then $R$ is regular. Since $R$ is a domain, the claim follows.
It remains to assume that $\dim R=2$ and $\mu(\fm)=3$. It follows that $\Ht(J)=1$.
Over $\UFD$, any height-one unmixed ideal is principal. Thus,
$R$ is hypersurface.  It remains to apply  Fact \ref{gor}.

ii)
The case $\mu(J)=1$ is in part i). We may assume that  $\mu(J)=2$.
Since $J$ is unmixed and $S$ is $\UFD$ we deduce that $\Ht(J)=2$. In particular,
$\grade(J,S)=\mu(J)=2$. This implies that $J$ is generated by a regular sequence of length two. Thus,
$R$ is complete-intersection. Now, the desired   claim  follows from   Fact \ref{gor}.
\end{proof}

\begin{corollary}\label{cmm2}Question 1.1  is not true
over a	Cohen-Macaulay local ring  $R$ of multiplicity two.
\end{corollary}

\begin{proof}
Recall from 	Abhyankar's inequality that
$\mu(\fm) - \dim R + 1 \leq \e(R)$. This  implies that
$\mu(\fm) \leq \dim R + 1$. Thus, 
$\hat{R}$ is hypersurface, and so the claim follows.
\end{proof}

By  $\HH^{i}_{\fa}(-)$ we mean the $i$-th $\it{\check{C}heck}$ cohomology module of $(-)$ with
respect to a generating set of $\fa$. 
Also, in the sequel we will use the concept of \it{limit closure}.
Let $\underline{y}:=y_1,\ldots,y_{\ell}\subset \fm$.
Recall that $R/(\underline{y})^{\lim}$ is the image
of $R/(\underline{y})$ under the isomorphism  $\HH^{\ell}_{\underline{y}}
(R)\cong{\varinjlim}_n R/(y^n_1, \ldots,  y^n_{\ell})$. 

\begin{example}
Let $R:=k[[x,y]]$. Then $(x^2,xy)^{\lim}=R$ and $(x,y)^{\lim}=(x,y)$. In particular, limit
closure does not preserve the inclusion.
\end{example}

\begin{proof}Set $J:=(x^2,xy)$ and $u:=x^2 xy=x^3y$. In the light of Hartshorne-Litchenbaum vanishing, we see $\HH^2_{J}(R)=0$. By definition, $J^{\lim}=R$. One may see this more explicitly: $1\in(J^{[3]}:_Ru^{2})= \left((x^{6} , x^3y ^3):_R  x^6y^{2}\right).$ 
 Since $x,y$ is a regular sequence, we have  $(x,y)^{\lim}=(x,y)$.
\end{proof}

However, by restriction over parameter ideals we have:

\begin{observation}
	Let $J_1\subset J_2$ be two ideals generated by a part of system of parameters.
	Then $J_1^{\lim}\subset J_2^{\lim}$.
\end{observation}

\begin{proof}
	The claim is trivial if the ring is Cohen-Macaulay, and we are going to reduce  to this case.
	Suppose $J_1=(y_1,\ldots,y_m)$. Set $y:=y_1.\cdots.y_m$. The sequence $\{
	(y_1^{n+1},\ldots,y_m^{n+1}):_Ry^{n}|n\in\mathbb{N}\}$ is increasing and its union is $J_1^{\lim}$.  There is an integer $n$ such
	that   $J_1^{\lim}=(y_1^{n+1},\ldots,y_m^{n+1}):y^{n}$.	By a theorem of Andr\'{e},  there is a big Cohen-Macaulay algebra $A$ over $R$.
	In fact,   $B:=\hat{A}$
	is balanced.   This yields that   $y_1,\ldots,y_m$
	is a regular sequence over $B$. Thus, $J_1^{\lim}=((y_1^{n+1},\ldots,y_m^{n+1}):_By^{n})\cap R=J_1 B\cap R.$ Since $B$ is balanced,	the same  argument implies that $J_2^{\lim}= J_2B\cap R$. 
	Since $J_1\subset J_2$, it follows  that $J_1^{\lim}\subset J_2^{\lim}$.
\end{proof}

This observation suggests:

\begin{definition}For an ideal $J$ of a local ring $R$, we set$$
J^{\mathrm{BCM}}:=\{x \in R~|~x \in JB~\mbox{for some balanced big Cohen-Macaulay}~R\mbox{-algebra}~B\}.
$$\end{definition}
	
Let $E_R(k)$
be the injective envelop of $k$ as an $R$-module. A local ring $R$ is called \it{quasi-Gorenstein} if $\HH^{\dim R}_{\fm}
(R)\cong E_R(k)$. 
 Here, we use a trick
that we learned from \cite{st}:

\begin{proposition}\label{quasi}
	 Let $(R,\fm,k)$ be quasi-Gorenstein,   $\fp\in\Assh(R)$ and let $u\in\fp$. 	Then $(0 : u)\nsubseteq I^{\lim}$. In particular,  $(0 : u)$ is not in any ideal generated by a system of parameters.
\end{proposition}

\begin{proof} 
	Set $A:=R/uR$ and  $d:=\dim R$. Since $u\in\fp\in\Assh(R)$, we deduce that $d=\dim A$.  Let $\fm=\fp_d\supset\ldots\supset\fp_0=\fp$ be a strict chain of prime ideals of $R$. By going down  property of flatness, there is a chain $\fm_{\hat{R}}=\fq_d\supset\ldots\supset\fq_0=:\fq$ of prime ideals of $\hat{R}$ such that $\fq_i$ lying over $\fp_i$.  In particular, there is a  $\fq\in\Assh(\hat{R})$ lying over $\fp$, and so $u\in \fq$. Since $(0 :_{\hat{R}} u)= (0 : u)\hat{R}$ and $I^{\lim}\hat{R}= (I\hat{R})^{\lim}$, without loss of the generality we may assume that    $R$ is complete, and $A$ is as well.
	
Let $\underline{y}:=y_1,\ldots,y_d$ be any system of parameter of $R$.	Let ${\underline{x}}$ be the lift of $\underline{y}$ to $A$. 
	Let $\mu_{\underline{x}}^A:A/{\underline{x}}A\to \HH^d_{{\underline{x}}}(A)$ be the  natural map.
	By the canonical element conjecture, which is now a theorem,  we have $\mu_{\underline{x}}^A\neq 0$.
 Since
	$R$ is quasi-Gorenstein, $\HH^d_{\fm}(R)=E_R(k)$. We set $(-)^v:=\Hom_R(-,E_R(k))$.
Denote the maximal ideal of $A$ by $\fn$.	 Grothendieck's
	vanishing theorem says that $\HH^{>d}_{\fn}(-)=0$. We apply this along with the independence theorem
	of local cohomology modules to observe that 
	$\HH^{d}_{\fn}(A)=\HH^{d}_{\fn}(R)\otimes_RA=\HH^{d}_{\fm}(R)\otimes_RA.$
	It turns out that $\mu_{\underline{x}}^A=\mu_{\underline{y}}^R\otimes_RA$.
	
	By definition, the map $R/(\underline{y})^{\lim} \to \HH^{d}_{\fm}(R)$ is
	injective, and its image is the submodule of $\HH^{d}_{\fm}(R)$ which
	is annihilated by $(\underline{y})^{\lim}$.
	 We have
	$$0\neq(\mu_{\underline{x}}^A)^v=\Hom_R(\mu_{\underline{x}}^A,E_R(k))=\Hom_R(\mu_{\underline{y}}^R\otimes_RA,E_R(k))=\Hom_R(A,(\mu_{\underline{y}}^R)^v)=
	\Hom_R(A, \sigma_y^R ),$$where
	$$\sigma_{\underline{y}}^R:=\left(R/(\underline{y}) \twoheadrightarrow R/(\underline{y})^{\lim}\stackrel{\cong}\lo\Ann_{E_R(k)}(\underline{y})^{\lim}\hookrightarrow E_R(k)\right)^v.$$ The assignment $E'\mapsto\Ann_R E'$ induces a 1-1 correspondence from submodules  $E'$ of $E_R(k)$  to ideals 
	of $R$. We have $E_R(k)^v=R$. The mentioned correspondence is given by $\ker(E_R(k)^v\twoheadrightarrow( E')^v)$ too.
	From these observations, for any ideal $J$ we have $(\Ann_{E_R(k)}J)^v\cong R/J$.
	In particular, 
	we can identify $\sigma_{\underline{y}}^R$, up to an isomorphism,  with the composition of $R\twoheadrightarrow R/(\underline{y})^{\lim}$
	and $ R/(\underline{y})^{\lim}\hookrightarrow (R/(\underline{y}))^v$. It follows that the composition
	$$\Hom_R(A,R)\lo\Hom_R(A,R/(\underline{y})^{\lim})\lo\Hom_R(A,(R/(\underline{y}))^v) $$is nonzero.
	In particular, the map  $\Hom_R(A,R)\lo\Hom_R(A,R/(\underline{y})^{\lim}) $ is nonzero.
	It turns out that $(0:u) \nsubseteq(\underline{y})^{\lim}$.
	To see the particular case, it is enough to note that $(\underline{y})\subset (\underline{y})^{\lim}$.
\end{proof}

\begin{remark}\label{quasin}
	The quasi-Gorenstein assumption is needed:
	Let $R$ be an equidimensional local ring  with zero-divisors equipped with a parameter ideal $I$ such that  $I^{\lim}=\fm$. Such a thing exists, see \cite[Example 6.1]{fh}. Let $\fp\in\Assh(R)$ and  let $u\in\fp$ be nonzero. Clearly, $(0 : u)\subset \fm=I^{\lim}$. In the next section we will show
	that  quasi-Gorenstein assumption is needed even in the particular case.
\end{remark}

\begin{observation}
	Let $R$ be of depth zero and $\fp\in\Assh(R)$. There is a nonzero $x\in\fp$
	such that $(0 : x)\nsubseteq  I$.		
\end{observation}

\begin{proof}Without loss of generality we  may assume that $\dim R>0$. Thus,
	there is $y\in\fm\setminus\{\fp\}$. Since $\depth(R)=0$, there is an $x$ such that $\fm=(0:x)$. Since $\fp$ is prime and  $xy=0$ we see $x\in\fp$.
	If $(0:x)\subset I$, 
	then we should have $\fm\subset I$. This implies that $R$ is regular, a contradiction.
\end{proof}

An $R$-module $K_R$ is called  canonical   if $K_R\otimes_R\hat{R}\cong\HH^{\dim R}_{\fm}(R)^v$. 
In the case the ring is Cohen-Macaulay, we denote it by $\omega_R$.

\begin{fact}\label{fk}A local
	ring is quasi-Gorenstein if  and only if $K_R$ (exists and)   becomes free and of rank one. 
\end{fact}

\begin{proposition}\label{pm} 
	Let $P$ be a nonzero prime ideal  of a local ring $(A,\fn)$. 
	Then  Question 1.1    is not true over $R:=\frac{A}{P\fn}$. 
	Also,   $R$ is  quasi-Gorenstein if and only if $\fn$ is principal.
\end{proposition}

\begin{proof}
	  Let $\fp:=PR$ (resp. $\fm:=\fn R$). 
	We have $\Assh(R)=\{\fp\}$. Let $u\in\fp$. We may assume that $u\neq 0$. Since $\fp \fm=0$ and $u\neq0$, it follows that
	$\fm=(0:u)$. If $(0:u)$ were be a subset of an ideal  $I$, generated by a parameter sequence,
	then we should have $\fm\subset I$. It turns out that $R$ is regular.
	But, $R$ is not even a reduced ring. This contradiction yields a negative answer to Question 1.1.

Suppose $\fn$ is principal.  Then $d:=\dim R\leq 1$. First, assume that $d=0$. Since $P\neq 0$, we have $P=\fn$.  
If $d=1$, since $\fn$ is principal, it follows that $R$ is a discrete valuation
 domain (see \cite[Theorem 11.7]{Mat}). Again, since $P\neq 0$ we deduce that $P=\fn$.
In each cases, $P=\fn$.
Thus,
socle of $R$ is $\fm$. Since $\mu(\fm)=1$, 
$R$ is quasi-Gorenstein. Now assume that  $\fn$ is  not principal. 
We have two possibilities: i) $P=\fn$, or ii) $P\neq \fn$. In the first case,
$\Soc(R)=\fm$. Since $\mu(\fm)=\mu(\fn) >1$,  $R$ is not quasi-Gorenstein.  Then, 
without loss of the generality we may and do assume that $P\neq \fn$.
 As ${\hat{R}}$ is complete, and in view of   \cite[(1.6)]{a0}, $K_{\hat{R}}$ exists.
 Recall that $\hat{R}\cong \frac{\hat{A}}{\hat{P}\hat{\fn}}$.
 We apply this along with $\Ass(K_{\hat{R}})=\Assh(\hat{R})$ (see \cite[(1.7)]{a0})
 to deduce that
$$\Ass(K_{\hat{R}})=\Assh(\hat{R})\subset\min(\hat{R})= \min{(\fp\hat{R})} \neq \{\min{(\fp\hat{R})},\fm\hat{R}\}\subset\Ass(\hat{R}).$$ 
Thus,  $K_{\hat{R}}$ is not  free. By Fact \ref{fk},   $\hat{R}$ is not quasi-Gorenstein. Recall  that  a local ring is 
quasi-Gorenstein if and only if  its completion is as well. So, $R$ is not quasi-Gorenstein.
\end{proof}

\begin{proposition}\label{pro}
	Let $P$ and $Q$ be two prime ideals of $S:= k[[X_1,\ldots,X_d]]$ generated by linear forms. 
	Let $R:=\frac{S}{PQ}$. Then Question 1.1    is not true over $R$.
\end{proposition}

\begin{proof}
	We may assume that neither $P$ nor $Q$ is zero.
Let $\fp$ (resp. $\fq$) be the image of  $P$ (resp. $Q$) in $R$. Then $\Assh(R)\subset\{\fp,\fq\}$.
Let $G(P)$ (resp. $G(Q)$) be the minimal monomial generating set of  $P$ (resp. $Q$).
Let $I$ be a parameter ideal.
Suppose first that $G(P)\cap G(Q)\neq \emptyset$. After rearrangement, we may assume that  $X_1\in G(P)\cap G(Q)$.
Also, without loss of generality, we set $G(Q):=\{X_1,\ldots,X_i\} $  and  $G(P):=\{X_1,\ldots\} $.
 By symmetry,  we may and do assume that $\fp\in\Assh(R).$  Let $u\in\fp$ and suppose
on the contradiction that $(0:u)\subset I$.
Since $\fp\fq=0$ we have  $\fq\subset (0:u)\subset I$. Recall that $\{x_1,\ldots,x_i\}\subset \fm$
  modulo $\fm^2$ is $k$-linearly independent. Since $I\fm\subset\fm^2$, we deduce that $\{x_1,\ldots,x_i\}\subset I$ 
  modulo $I\fm$ is $k$-linearly independent.  In particular, $\{x_1,\ldots,x_i\}$ is part of  a minimal
   generating set of $I$. From this, $x_1$ is a parameter element. Thus, 
   $x_1\notin\bigcup_{\mathfrak r\in\Assh(R)}\mathfrak r\subset \fp\cup \fq.$ This is a contradiction. 
Then, without loss of generality we may assume that   $G(P)\cap G(Q)=\emptyset$.
After rearrangement,
we can assume that $G(P):=\{X_1,\ldots,X_m\} $  and  $G(Q):=\{X_{m+1},\ldots,X_{m+n} \} $.
We take $\ell:=d-(m+n)$.
   Also, by symmetry,  we may  assume that $m\leq n$.  We have two possibilities: i) $m< n$,  or  ii)  $m= n$.

i) Since  $m< n$, we have  $\Assh(R)=\{\fp\}$. Let $u\in\fp$ and suppose
on the contradiction that $(0:u)\subset I$. 
We have  $\fq\subset (0:u)\subset I$. It turns out that $\{x_{m+1},\ldots,x_{m+n}\}$ is part of 
 a minimal generating set of $I$, and so  part of a system of parameters.  In view of
  Fact \ref{well} we see: $$\ell+m=\dim S/Q=\dim R/\fq=\dim R/ (x_{m+1},\ldots,x_{m+n})=\dim R-n=(\ell+n)-n=\ell.$$
  This implies $m=0$, and consequently $P=0$. This is a contradiction.

ii) The condition  $m= n$  implies that  $\Assh(R)=\{\fp,\fq\}$. The same  argument as i) yields the desired claim.
\end{proof}

The behavior of quasi-Gorenstein property under certain flat ring extensions
is subject of \cite{a0}. 

\begin{lemma}\label{lp}
	Let $A$ be a complete  
local ring. Then $A$ is  quasi-Gorenstein  if and only if $A[[X]]$ is  quasi-Gorenstein.
\end{lemma}

\begin{proof}By Cohen's structure theorem, $A$ is quotient of a Gorenstein local ring $G$.
	 Also, $A[[X]]$  is quotient of a Gorenstein local ring $G[[X]]$. In particular, 
$K_A$ and $K_{A[[X]]}$ exist, see \cite[(1.6)]{a0}.
Set $r:=\dim G-\dim A=\dim( G[[X]])-\dim( A[[X]])$. In view
of \cite[(1 .6 )]{a0}  we have $$K_{A[[X]]}\cong\Ext_{G[[X]]}^r(A[[X]],G[[X]]) \cong \Ext_G^r(A,G)\otimes_AA[[X]]\cong K_A\otimes_AA[[X]]\quad(\ast)$$ 
Suppose  $A[[X]]$ is      quasi-Gorenstein.  By  applying $-\otimes_{A[[X]]}\frac{A[[X]]}{(X)}$ 
along with $A[[X]]\cong K_{A[[X]]}\stackrel{(\ast)}\cong K_A\otimes_AA[[X]]$
we deduce that  $A \cong K_A.$  By Fact \ref{fk}   $A$ is      quasi-Gorenstein. 
 The converse part follows  by $(\ast)$.
\end{proof}

 The following result inspired from \cite{h}.

\begin{corollary}\label{ap}
	Let $P$ and $Q$ be two nonzero prime ideals of $S:= k[[X_1,\ldots,X_n]]$ generated by
	 linear forms such that $G(P)\cap G(Q)= \emptyset$ and let $R:=\frac{S}{PQ}$.  The following are equivalent:
		\begin{enumerate}
			\item[i)]   $R$  is hypersurface,
			\item[ii)] $R$ is complete-intersection,
			 	\item[iii)] $R$ is Gorenstein,
			 		\item[iv)] $R$ is  quasi-Gorenstein,
			 		\item[v)] $R$ is  Cohen-Macaulay.
		\end{enumerate}  
\end{corollary}

\begin{proof} 
	First, we prove that the  first four items are equivalent. Among them, the only nontrivial implication 
	is $iv) \Rightarrow i)$: We assume that $R$ is quasi-Gorenstein.  Suppose on the way of contradiction
	 that  one of $Q$ and $P$  is not principal.   By symmetry, we may and do assume that $P$ is 
	 not principal.  Since $G(P)\cap G(Q)= \emptyset$, both of $PR$ and $QR$ are minimal prime ideals of $R$.
	   Recall that quasi-Gorenstein rings are equidimensional.   It turns out that $\Ht(P)=\Ht(Q)$.
Let $G(P)$ (resp. $G(Q)$) be the minimal monomial generating set of  $P$ (resp. $Q$).	
Without loss of generality, we set $G(Q):=\{X_1,\ldots,X_{\ell}\} $  and  $G(P):=\{X_{\ell+1},\ldots,X_{2\ell}\} $.
Since   $P$ is not principal,  $\ell\geq2$. The extension
$\frac{k[[X_1,\ldots,X_{2\ell}]]}{(X_1,\ldots,X_{\ell})(X_{\ell+1},\ldots,X_{2\ell})}\lo R$
is either the identity map or is the power series extension. Then,  in view of Lemma \ref{lp}, 
 we may and do assume that $2\ell=n$.
For each $i\leq \ell$, we set $a_i:=x_i+x_{\ell+i}$ and we denote the ideal generated by them with $I$.
Since $x_{i}a_i=x_{i}^2$ and $x_{\ell+i}a_i=x_{\ell+i}^2$ we deduce that $\{a_i\}$ is a system of parameters. 
We set $\zeta:=\prod_{i=1}^{\ell} a_i$. Then $\fm\zeta\subset(a_i^2)_{i=1}^{\ell}$ (here, we need $\ell\geq 2$). 
We apply this along with
$I^{\lim}=\bigcup_{n\geq 1}((a_1^{n},\ldots,a_{\ell}^n):\zeta^{n-1})$ to deduce that  $I^{\lim}=\fm$. 
In the light of  Proposition \ref{quasi}  we see that $R$ is not quasi-Gorenstein. 
This is a contradiction that we searched for it.

Here, we show $iv) \Rightarrow v)$: It is enough to use $iv) \Leftrightarrow i)$.

Finally, we show $v) \Rightarrow i)$: As $R$ is  equidimensional
and by using the above argument, we deduce that
$R_0:=\frac{k[[X_1,\ldots,X_{2\ell}]]}{(X_1,\ldots,X_{\ell})(X_{\ell+1},\ldots,X_{2\ell})}\to R$
is either the identity map or is the power series extension.  It follows that
$R_0$ is Cohen-Macaulay. We claim that depth of  $R_0$ is one.  The element $x_1+x_{\ell+1}$ is regular,
 because $\zd(R_0)=\bigcup_{\fp\in\Ass(R_0)}\fp=(x_1,\ldots,x_{\ell})\cup(x_{\ell+1},\ldots,x_{2\ell})$. We need to show $\depth(R_0)\leq 1$.
In view of \cite[Corollary 3.9]{41}, a  way  to see this, is that its punctured spectrum is disconnected.  
The closed subsets $\V(x_1,\ldots,x_{\ell})\setminus\{\fm\}$ and  $\V(x_{\ell+1},\ldots,x_{2\ell})\setminus\{\fm\}$
 are disjoint, non-empty   and their union is $\Spec^{\circ}(R_0)$. This says that $\Spec^{\circ}(R_0)$ is 
 disconnected.  So, $\depth(R_0)=1$. Since $R_0$   is Cohen-Macaulay, $\ell=\dim R_0=\depth(R_0)=1$. 
 From this, $R$ is hypersurface.
\end{proof}

\begin{proposition}\label{power}
	Let $P$  be a prime ideal of $S:= k[X_1,\ldots,X_d]$ generated by linear forms and $R:= S/P^n $
	 for some $n>0$. Let   $u\in \fq\in\Assh(R)$.   Then $(0 : u)$  is not in an ideal $I$ generated
	  by a homogeneous system of parameters.
\end{proposition}

\begin{proof}
	Without loss of generality, we assume that  $P\neq 0$ and $n>1$, because the claim is clear over integral domains. After rearrangement,
	$X_1\in P$.
	Let $\fp:=PR$. Then $\Assh(R)=\{\fp\}$, i.e., $\fq=\fp$.
	Suppose
	on the contradiction that $(0:u)\subset I$.
	Since $\fp^n=0$ we have  $x_1^{n-1}\subset (0:u)\subset I$. 
	Let $i$ be the smallest integer such that $x_1^{i}\in I$.  Then $i\leq n-1$.
	First, we deal with the case $x_1^{i}\in \fm I$.
	There are $a_j\in \fm$ and $b_j\in I$ such that $x_1^i=\sum_j 	a_jb_j$. By  looking at this equation in $S$ we get 
	$X_1^i-\sum _jA_jB_j\in P^n.$ Since $S$ is $\UFD$ and by a degree-consideration, there is an $\ell>0$
	such that   $X_1^{\ell}=A_j$ and $B_j=X_1^{i-\ell}$ for some $j$.
	Since $b_j\in I$, we see that  $x_1^{i-\ell}\in I$. This is impossible,
	because of the minimality of $i$. This implies that $x_1^{i}\notin \fm I$.  In particular,	$x_1^i$ is  
	 a  parameter element, because it is part of a minimal generating set of $I$. So,  $x_1^i\notin\bigcup_{\fq\in\Assh(R)}\fq=\fp,$  a contradiction.
	The proof is now complete.
\end{proof}

\begin{remark}Adopt one of the following situations:
	\begin{enumerate}
		\item[i)]  	Let  $S:= k[[X_1,\ldots,X_d]]$ and $\fp$  be a prime ideal generated by
		linear forms.
		\item[ii)]  Let $S$ be a $4$-dimensional unramified complete regular local ring and  $\fp$ be any prime ideal. 
	\end{enumerate}  
	Let $R:= S/\fp^n$ for some $n>1$. Then $R$ is  quasi-Gorenstein if and only if $\fp$ is principal.
\end{remark}

\begin{proof}The if part is clear.
	Now, suppose $\fp$ is not principal.

i)
 After rearrangement, there is an $0<\ell<d$ such that $\fp=(X_i)_{i=\ell}^d$. 
		Set $A:=\frac{k[[X_{\ell},\ldots,X_d]]}{(X_i|  \ell\leq i\leq d)^n}$. Its socle is not principal.
		 Thus, $A$ is not Gorenstein.  Recall that
		$  R=A[[X_1,\ldots,X_{\ell-1}]]$. In view of Lemma \ref{lp}
		we deduce that $R$ is not quasi-Gorenstein.
	
	ii)  Suppose on the way of contradiction that $R$ is quasi-Gorenstein, i.e., $R=K_R$.
	  Recall that $K_R$ satisfies Serre's condition $S(2)$.
		Since $\fp$ is not principal, $\Ht(\fp)\geq 2$, i.e., $\dim R\leq 2$. From these, 
		$R$ is Cohen-Macaulay. It follows that $R$ is Gorenstein.
		By \cite{h} this is impossible.	
\end{proof}

In the same vein we have:

\begin{example} Let $S:=k[X_1,\ldots,X_m]$ be a polynomial ring, $\fn$ be its irrelevant ideal, and
	let $R:=\frac{S}{P\fn^2}$ for some homogeneous prime ideal 
	$P$ of $S$ containing a linear form.  Let   $\fq\in\Assh(R)$ and take $u$ be  in $\fq$. 
	  Then $(0 : u)$  is not in an ideal $I$ generated by a homogeneous system of parameters.
\end{example}


\section{ Mores on non   quasi-Gorenstein  rings}

The following yields
	another proof of Corollary \ref{ap}, because  $PQ=P\cap Q$.
	
	\begin{observation}\label{inter}
		Let $(A,\fn)$ be   Cohen-Macaulay. Let $I$ and $J$ be two unmixed ideals of $A$ of same height and  
		 $\Ht(I+J)\geq\Ht (I)+2$. Then $R:=\frac{A}{I\cap J}$ is not quasi-Gorenstein.	In fact $\HH^{\dim R}_{\fm}(R)$ decomposable.
	\end{observation}
	
	\begin{proof}
	 Since $A$ is Cohen-Macaulay, $I$ and $J$ are unmixed and of same height we deduce
	  that $d:=\dim R=\dim A/I=\dim A/J$. Similarly, $\dim(A/I+J)\leq d-2$, because
	  $\Ht(I+J)\geq\Ht (I)+2$. We use Grothendiek's vanishing theorem along with a 
	  long exact sequence  of local cohomology modules induced by $0\to R \to A/I\oplus A/J\to A/(I+J)\to 0$ 
	    to find a   decomposition of $\HH^d_{\fm}(R)\cong\HH^d_{\fm}(A/I)\oplus \HH^d_{\fm}(A/J)$. 
	     By Grothendieck's non-vanishing theorem, the decomposition is nontrivial. Due to flat base change theorem,
	      we know completion behaves well with local cohomology modules. We apply Matlis' functor over $\hat{R}$
	       to see that $K_{\hat{R}}$ equipped with a nontrivial decomposition. In particular,
	        $K_{\hat{R}}$ is not of rank one. Thus, $\hat{R}$ is not quasi-Gorenstein. So,  $R$ is not quasi-Gorenstein. 
	\end{proof}

We left to the reader to deduce the third proof of Corollary \ref{ap}  from the following result that its proof is more technical than Observation \ref{inter}:

\begin{fact}(Hochster-Huneke)\label{ho}
Let $R$ be $d$-dimensional local, complete and equidimensional. Then $\HH^d_{\fm}(R)$ is indecomposable if
for any $\fp,\fq\in\min(R)$ there are minimal prime ideals $\fp_0:=\fp,\ldots,\fp_n:=\fq$ such that $\Ht(\fp_i+\fp_{i+1})\leq 1$ for all $i$.	
	\end{fact}

As complete  rings are catenary, the following may be considered as a slight generalization of Fact \ref{ho}.

\begin{corollary}\label{cat}Let $A$ be catenary
	and equidimensional. Let $\fp$ and $\fq$ be  prime ideals of $A$ of same height and   
	$\Ht(\fp+\fq)\geq\Ht (\fp)+2$.  Set $R:= \frac{A}{\fp\cap \fq}$. Then $\HH^{\dim R}_{\fm}(R)$ decomposable. 
	 In particular, $R$ is not quasi-Gorenstein.
\end{corollary}
 
\begin{proof}
	The assumptions guarantee that $d:=\dim R=\dim A/\fp=\dim A/\fq=\dim A-\Ht(\fp)$, see  \cite[\S 31, Lemma  2]{Mat}.
	Also, we have $\dim(A/\fp+\fq)\leq \dim A-\Ht(\fp+\fq)\leq d-2$. By
the proof  of	Observation \ref{inter}  we get the claim.  
\end{proof}

\begin{proposition}\label{catg2}Let $(R,\fm)$ be catenary
		and equidimensional. Let $\fp_1,\ldots,\fp_n$ be  prime ideals of $R$ of same height, $n>1$ and  
		 $\Ht(\fp_i+\fp_{i+1})\geq\Ht (\fp_i)+2$ for all $1\leq i<n$. Set $d:=\dim \frac{R}{\fp_1\ldots\fp_n}$.
		  Then $\HH^d_{\fm}(\frac{R}{\fp_1\ldots\fp_n})$ decomposes into $n$ nonzero submodules.  
		   In particular, $\frac{R}{\fp_1\ldots\fp_n}$ is not quasi-Gorenstein.
	\end{proposition}

\begin{proof} We argue by induction on $n\geq 2$. First, we deal with the case $n:=2$
	and for simplicity we set $\fp:=\fp_1$ and $\fq=\fp_2$.
The assumption $\Ht(\fp+\fq)\geq\Ht (\fp)+2$ implies that $\fp\neq \fq$.
Since $\fp$ and $\fq$ are of same height, $\fq\nsubseteqq \fp$. Thus, $\fq R_{\fp}=R_\fp$. 
Also, flat extensions behave well with respect to the intersection of ideals. Consequently, $(\frac{\fp\cap \fq}{\fp\fq})_\fp=0$.
This yields that
 $\dim(\frac{\fp\cap\fq}{\fp\fq})<\dim(\frac{R}{\fp\fq})$. By Grothendiek's vanishing theorem, $\HH^{>d-1}_{\fm}(\frac{\fp\cap \fq}{\fp\fq})=0$.
		We look at the short exact sequence   $0\to \frac{\fp\cap \fq}{\fp\fq}\to \frac{R}{\fp\fq}\to\frac{R}{\fp\cap \fq}\to 0$.
		 This induces the following exact sequence
	$$0= \HH^d_{\fm}(\frac{\fp\cap \fq}{\fp\fq})\lo\HH^d_{\fm}(\frac{R}{\fp\fq})\lo 
	\HH^d_{\fm}(\frac{R}{\fp\cap \fq})\lo\HH^{d+1}_{\fm}(\frac{\fp\cap \fq}{\fp\fq})=0.$$
We plug this in  Corollary \ref{cat} to get a nontrivial decomposition $\HH^d_{\fm}(\frac{R}{\fp\fq})\cong
 \HH^d_{\fm}(\frac{R}{\fp\cap \fq})\cong\HH^d_{\fm}(\frac{R}{\fp})\oplus \HH^d_{\fm}(\frac{R}{\fq})$.  This completes
the proof when $n=2$. Now suppose, inductively, that $n\geq 3$, and the result has been
proved for  $n-1 $. By repeating the above   argument, we see  $ \HH^d_{\fm}(\frac{R}{\fp_1\ldots\fp_n})\cong
 \HH^d_{\fm}(\frac{R}{\fp_1\ldots\fp_{n-1}})\oplus \HH^d_{\fm}(\frac{R}{\fp_n})$. By  the inductive step,  
 $\HH^d_{\fm}(\frac{R}{\fp_1\ldots\fp_n})\cong\bigoplus_{i=1}^n\HH^d_{\fm}(\frac{R}{\fp_{i}})$. 
In the light of Grothendieck's non-vanishing theorem, $\HH^d_{\fm}(\frac{R}{\fp_{i}})\neq 0$. 
This completes the proof.
\end{proof}

\begin{corollary}\label{g}Let $R$ be any  noetherian local ring. Let $\fp_1,\ldots,\fp_n$ be 
	 prime ideals of $R$ of same codimension, $n>1$ and   $\dim(\frac{R} {\fp_i+\fp_{i+1}})\leq\dim \frac{R} {\fp_i} -2$
	  for all $1\leq i<n$. Then $\frac{R}{\fp_1\ldots\fp_n}$ is not quasi-Gorenstein.
\end{corollary}

\begin{example}
	The bound  $\Ht(\fp_1+\fp_2)\geq\Ht (\fp_1)+2$ is sharp. It is enough to look at $\fp_1:=(x)$ and $\fp_2:=(y)$ in $R:=k[[x,y]]$. 
\end{example}

\begin{corollary}\label{gcm}Let $R$ be any  noetherian local ring. Let $\fp,\fq$ be  prime ideals of $R$
	 of same codimension such that  $\dim(\frac{R} {\fp+\fq})\leq\dim \frac{R} {\fp} -2$. 
	 Then $\frac{R}{\fp\fq}$ is not Cohen-Macaulay.
\end{corollary}

\begin{proof}
On the way of contradiction we assume that	$\frac{R}{\fp\fq}$ is   Cohen-Macaulay.
Let $d:=\dim \frac{R}{\fp\fq}$. By Corollary \ref{g}, $\HH^d_{\fm}(\frac{R}{\fp\fq})$ equipped 
with a nontrivial decomposition.
The same thing holds for $\HH^d_{\fm}(\widehat{\frac{R}{\fp\fq}})$. By Matlis
duality, $\omega_{\widehat{\frac{R}{\fp\fq}}}$ decomposes into nontrivial submodules. This is a contradiction.	
\end{proof}

\begin{corollary} 
	Adopt the notation of 	Corollary  \ref{g}. Then $\frac{R}{\fp_1\ldots\fp_n}$ is not  Cohen-Macaulay.
\end{corollary}

	\begin{corollary} 
	Adopt the notation of Observation \ref{inter}. Then $R:=\frac{A}{I\cap J}$ is   not Cohen-Macaulay.
\end{corollary}

\begin{proof}
Combine  Observation \ref{inter} along with the argument of Corollary \ref{gcm}.
\end{proof}


\section{Negative side of Question 1.1}

Over a Cohen-Macaulay local ring $(R, \fm)$ we have 
$\mu(\fm) - \dim R + 1 \leq \e(R)$. If the equality holds we say $R$ is of minimal multiplicity.
Here,  we show   the multiplicity  two  (resp. quasi-Gorenstein) assumption of Corollary \ref{cmm2} (resp. Proposition \ref{quasi})  is important.
Also,   both  assumptions $\dim S<4$ and
$\mu(J)<3$ (resp. 1-dimensional assumption) of Observation \ref{em} (resp. Fact \ref{1}) are really needed.

\begin{example}	\label{nred}For each $n>0$, set $	R_n:= \frac{\mathbb{Q}[X, Y, Z, W]_{\fm}}{(XY - ZW, W^n, YW)}$.
	\begin{enumerate}
		\item[i)]    Question 1.1  has negative answer over $R_n$ if and only if $n=1$.
	\item[ii)] $R_n$  is Cohen-Macaulay if and only if $n<3$.
		\item[iii)] $R_2$ is two-dimensional,
		generically Gorenstein, Cohen-Macaulay, almost complete-intersection, of type two and of minimal multiplicity $3$.
	\end{enumerate}  
\end{example}

\begin{proof}
The   ring $R_1$ is  hypersurface. By Fact \ref{gor}, we get the claim for $n=1$. Then we may assume that $n>1$. We set $u:=x$ and  $\fp:=(x,w)$. 
Recall that $xy\in(z)$, and so $xy\in\rad(x+y, z)$. Since $x^2=x(x+y)-xy=x(x+y)-zw$
	we see $x\in\rad(x+y, z)$. Similarly, $y\in\rad(x+y, z)$. 
		Also, $w\in\rad(x+y, z)$, because it is nil.
Hence, $\rad(x+y,z)=\fm$. Since $\Ass(R_2)=\{\fp, (w ,y)\}$, $\dim (R_n)=2$.
 Thus $ \{x+y, z\}$	is a parameter sequence. 
 
  Set $P:=\mathbb{Q}[X, Y, Z, W]_{\fm}$. The free 
   resolution of $R_2$ over $P$ is  $0\to P^2\stackrel{A}\lo P^3\to P\to R_2\to 0 ,$
  where \begin{equation*}	A:= \left(
  \begin{array}{cccc}
  -W	& 0     \\
  X 	&	-W     \\
  -Z 	& Y   
  \end{array}  \right) .
  \end{equation*}
Since $\pd_P( R_2)=2$, and  in view of Auslander-Buchsbaum formula, we deduce that $R_2$
	is Cohen-Macaulay.
	Suppose $n>2$. The primary
	decomposition of $J:=(XY - ZW, W^n, YW)$
	is given by $$(W, X)\cap (W^2 , Y W, Y^2 , X Y - Z W)\cap(Z, Y, W^n ).$$
	Then,
 $\Ass(R_n)=\{\fp, (w,y),  (w,y,z)\}$. Since
	$R_n$ has an embedded prime ideal, it is not Cohen-Macaulay. Recall that $\fm\notin\Ass(R_n)$
	we deduce that $\depth(R_n)=1$.

From 	$xy^2=xy^2-yzw=	y(xy - zw)=0$,  we conclude that $y^2\in (0:u)$. We claim  that $(y^2)= (0:u)$.
If $n=2$ this follows from the primary decomposition $(0)=(y^2)\cap (x,w)$. Now, let $n>2$.
The only primary components of $J$ that contains $X$ is $(W, X)$.
Now we compute the intersection
of reminder: $$I:=(W^2 , Y W, Y^2 , X Y - Z W)\cap(Z, Y, W^n )=(W^n , Y W, Y^2 , X Y - Z W).$$
Since $w^n=yw=xy-zw=0$,  the image of $I$ in $R_n$ is $(y^2)$. From this $(0:x)=(y^2)$.
We conclude from $y^2=y(x+y)-yx=y(x+y)-wz$ that $(0:u)$ is   in an ideal generated by a system of parameters.

Since  $\omega_{R_2}=(y,w )R_2$ we know that   $R_2$ is  
	generically Gorenstein and of type two. 
Due to the equality $\fm^2=(x+y,z)\fm$ we remark that $(x+y,z)R_2$ is a reduction of $\fm$.
It turns out that
$\e(R_2)=\e(x+y, z;R_2)$. The  chain  $(x+y,z)\subset(x+y,z,y)\subset(x+y,z,y,w)\subset R_2$ shows   that $\ell(R_2/(x+y, z))=3$. 
Since
$R_2$ is  Cohen-Macaulay, $\e(R_2)=\e(x+y, z,v+w;R_2)=\ell(R_2/(x+y, z,v+w))=3.$
In particular, $R_2$ is of minimal multiplicity.
\end{proof}

Here, we present  an  example of multiplicity two. In particular, the Cohen-Macaulay  
assumption of Corollary \ref{cmm2}   is important.

\begin{example}	\label{red} Let $R:= \frac{\mathbb{Q}[ X, Y, Z, W,V]_{\fm}}{(XY - ZW, WV, YW)}$. Then 
	   Question 1.1  has positive answer over $R$ for certain $\fp$ and $u$. Also,
$\e(R)=2$ and $\depth(R)=2<3=\dim R$.
\end{example}

\begin{proof}
Recall that $xy\in(z)$, and so $xy\in\rad(x+y, z)$. Since $x^2=x(x+y)-xy$
we see $x\in\rad(x+y, z)$, and $y\in\rad(x+y, z)$.
In view of $v^2=v(v+w)$ we see $v\in\rad(x+y, z,v+w)$. In the same vein, $w^2\in (x+y, z,v+w)$.  In sum, $\rad(x+y,z,v+w)=\fm$.
In order to show $ \{x+y, z,v+w\}$	is a parameter sequence, we remark that $\dim R=3$. To see this,  we recall
 $\Ass(R)=\{(v, z, y),  (w, x),   (w, y),  (v, w, y)\}$.  In fact, the primary
decomposition of $J:=(XY - ZW, WV, YW)$
is given by $$(V, Z, Y)\cap (W, X)\cap(W, Y)\cap (V, W^2 , YW, Y^2 , XY - ZW).$$
 Since
$R$ has an embedded associated prime ideal, it is not Cohen-Macaulay.   Set $P:=\mathbb{Q}[X, Y, Z, W,V]_{\fm}$.
 The projective 
resolution of $R$,  as a $P$-module, is given by  $0\to  P\stackrel{B}\lo P^3\stackrel{A}\lo P^3\to P\to R \to 0 ,$ 
where \begin{equation*}	A:= \left(
\begin{array}{cccc}
0	& -YW    & -WV  \\
-V 	&XY-ZW   & XV  \\
Y 	&  0     & -ZW   
\end{array}  \right)\quad and\quad B:= \left(
\begin{array}{cccc}
ZW	     \\
-V 	    \\
Y 	 
\end{array}  \right).
\end{equation*}
Since $\pd_P(R)=3$
we deduce that $\depth(R)=2$.
 We set $u:=x$ and  $\fp:=(x,w)$. 
From 	$xy^2=xy^2-yzw=	y(xy - zw)=0$,  we conclude that $y^2\in (0:u)$.
From 	$xyv=xyv-zvw=	v(xy - zw)=0$,  we conclude that $yv\in (0:u)$.
Thus, $(y^2,yv)\subset (0:u)$. We are going to show the reverse inclusion.
The only primary components of $J$ that contains $X$ is $(W, X)$. Now we compute the intersection
of reminder:
 $$I:=(V, Z, Y)\cap(W, Y)\cap (V, W^2 , YW, Y^2 , XY - ZW)=(WV, YV, YW, Y^2 , XY - ZW).$$
Since $wv=yw=xy-zw=0$,  the image of $I$ in $R_n$ is $(y^2,yv)$. From this $(0:x)=(y^2,yv)$.

We conclude from $y^2=y(x+y)-wz$ that $y^2\in(x+y,z)\subset (x+y,z,v+w)$.
Since $yw=0$ we have $$yv=y(v+w)-yw=y(v+w)\in(v+w)\subset (x+y,z,v+w).$$
These observations yield 
 that $(0:u)\subset(x+y,z,v+w)$. The later  is generated by a system of parameters.

It is easy to see that $\fm^2=(x+y, z,v+w)\fm$. By definition,
$(x+y, z,v+w)$ is a reduction of $\fm$.  Recall that
$\e(R)=\e(x+y, z,v+w;R)$. The following chain $$(x+y, z,v+w)\subset(x+y, z,v+w,y)\subset(x+y, z,v+w,y,w)\subset R$$shows  
  that $\ell(R/(x+y, z,v+w))=3$. 
Since
$R$ is not Cohen-Macaulay, $\e(x+y, z,v+w;R)<\ell(R/(x+y, z,v+w))$. 
Note that $\Assh(R)$ is not singleton. We put this
along with the associativity formula for Hilbert-Samuel multiplicity
to deduce that $\e(R)\neq1$. In view of $$2\leq \e(R)=\e(x+y, z,v+w;R)<\ell(R/(x+y, z,v+w))=3,$$ we deduce that $\e(R)=2$.
\end{proof}

The above ring is not reduced: $(zw)^2=0$. 


\section{A remark  on the union of parameter ideals}

We denote the family of all ideals generated by a system of parameters by $\Sigma$.
We are interested in $\bigcup_{I\in\Sigma}I$.
	Parameter ideals  may have nontrivial nilpotent elements. This may happen even over Cohen-Macaulay rings. For instance  over $R:=\frac{k[[X,Y]]}{(X^2)}$
	the nilpotent element $xy$ is in the parameter ideal $(y)$.  More generally:
	
\begin{remark}
		Let $R $ be a local ring of positive  depth. If $\bigcup_{I\in\Sigma}I$
	has  no nontrivial nilpotent elements, then $R$ is reduced.	Indeed, let $0\neq y\in\nil(R)$, and let $x$ be a regular element.
	By extending $x$ to a system of parameters, we see $x\in\bigcup_{I\in\Sigma}I$.
	So, $xy\in\bigcup_{I\in\Sigma}I$. Since $x$ is regular, $xy\neq0$ and it is nilpotent.
	Since $(\bigcup_{I\in\Sigma}I)\cap\nil(R)=0$, we get to a contradiction.
\end{remark}

The above depth condition is important:

\begin{example} 
	  Let $R:=\frac{k[[X,Y]]}{X(X,Y)}$. 
Then $\nil(R)\neq 0$ and $(\bigcup_{I\in\Sigma}I)\cap\nil(R)=0$.
\end{example}

\begin{proof}
Clearly, $\nil(R)=(x)\neq 0$. Let $f\in I$ be  nilpotent for some  $I\in\Sigma$. We have $f=a_{10}x+a_{01}y+a_{02}y^2+\cdots$.
	We set $a:=a_{01}+a_{02}y+a_{03}y^2+\cdots$.	Then, $f=a_{10}x+ya$, since
		 $x^2=xy=0$. From these $f^2=y^2a^2$. There is an $n\in2\mathbb{N}$ such that  $0=f^n=y^na^n$.
	As $\Ann(y^n)=(x)$ we deduce that $a^n\in(x)$. As $(x)$ is prime,  $a\in(x)$. Consequently, $f=a_{10}x+ya=a_{10}x$.
	It is enough to show $a_{10}=0$. On the way of contradiction we assume that $a_{10}\in k^{\ast}$.
	We conclude that $x\in I$. Let $g\in\fm\setminus(x)$ be such that $I=(g)$. Let $c$ and $d$ be such that $g=cx+dy$.
Take $r$ be such that $x=rg$. Since $g\notin(x)$, we have $r\in(x)$. Let $s$ be such that $r=sx$. Therefore $x=rg=sxg=sx(cx+dy)=0$, a contradiction.
\end{proof}

\begin{proposition}\label{int} Let $I\in\Sigma$ and 
		let $R$  be one of the following three classes of  local rings:  
	 i)    quasi-Gorenstein,
		 ii) a Cohen-Macaulay ring of dimension one,  or
		 iii) a Cohen-Macaulay ring	of multiplicity two.
	Then $R$ is an integral domain if and only if 	 $Q\subset I$ for some $Q\in\Spec(R)$.
\end{proposition}

\begin{proof}
	The only if part is trivial. Conversely, assume that $Q\subset I$ for some $Q\in\Spec(R)$.
	Let $Q_0\subset Q$ be a minimal prime ideal. Since $R$ is equidimensional, $Q_0\in\Assh(R)$.
	By definition, there is an $x\in R$ such that $(0:x)=Q_0$. Suppose on the way of contradiction that $Q_0\neq 0$.
	Since $xQ_0=0$ we have  $x\in\zd(R)=\cup_{\fq\in\Ass(R)}\fq=\cup_{\fq\in\Assh(R)}\fq$.
	There is $ \fq\in\Assh(R)$ such that $x\in\fq$. Recall that $(0:x) =Q_0\subset Q\subset I$. 
	  In the case  i),   Proposition \ref{quasi}  lead us to a contradiction. In the case ii)  
	 (resp.  iii) it is enough to apply Fact \ref{gor} (resp.  Corollary \ref{cmm2}).
\end{proof}

\begin{corollary}\label{intC} 
 Let $(R,\fm)$ be as Proposition \ref{int}  and assume in addition that $R$ has a prime element (e.g., $R$ is hypersurface).
	Then $R$ is an integral domain if and only if 	 $\fm=\bigcup_{I\in\Sigma}I$.
\end{corollary}

\begin{proof}	The only if part is trivial. Conversely, assume that $\fm=\bigcup_{I\in\Sigma}I$.
	Let $p\in\fm$ be a prime element.    Then $p\in I$ for some $I\in\Sigma$.
Since the ideal $(p)$ is prime, the desired claim is in Proposition \ref{int}.
\end{proof}

We denote the set of all unit elements of $(-)$ by $\U(-)$.
The following example presents  a connection from $\bigcup_{I\in\Sigma}I$ to $\U(-)$:

\begin{example}
	Let $R:=k[[X,Y]]/(X^2)$. Then $\fm\setminus\bigcup_{I\in\Sigma}I=\U(R)x\simeq_{set} \U(R)$.
\end{example}

\begin{proof}Clearly, $\U(R)x\subset \fm\setminus\bigcup_{I\in\Sigma}I$. For the reverse inclusion, let
	$f\in \fm\setminus\bigcup_{I\in\Sigma}I$. There are $a_{ij}\in k$
	 such that $f=a_{10}x+a_{01}y+a_{11}xy+a_{02}y^2+\cdots$. 
	We set $a:=a_{01}+a_{11}x+a_{02}y+\cdots$.	Then, $f=a_{10}x+ya$. 
	We use $f\notin \bigcup_{I\in\Sigma}I$ and the fact that $ya$ is in the  parameter ideal $(y)$ to conclude 
	$a_{10}\neq0$.
Recall that  $f$ is not   a parameter element. We apply this to see
	$f\in\bigcup_{\fp\in\Assh(R)}\fp=(x)$. We plug this in $f=a_{10}x+ya$ to observe that $ya\in(x)$.
	Since $y\notin(x)$, we have $a\in(x)$. Let $r\in R$ be such that $a=rx$.
	Then $f=a_{10}x+ya=a_{10}x+yrx=x(a_{10}+ry)\in x \U(R),$ 
	because $a_{10}\in  \U(k)$. 
\end{proof}

\begin{acknowledgement}
	We used Macaulay2 several times.
\end{acknowledgement}

\end{document}